\documentclass[11pt]{article}

\usepackage{amstext, amsmath,latexsym,amsbsy,amssymb}
\usepackage{enumitem}
\usepackage{amsfonts}
\usepackage{amsmath,amssymb}
\usepackage{color}
\usepackage[pdftex]{graphicx}
\usepackage{multirow}
\usepackage{latexsym}
\usepackage{mathrsfs}
\usepackage{bm}

\newtheorem{remark}{Remark}[section]
\newenvironment{proof}{{\bf Proof\ }}{\QED\\}

\numberwithin{equation}{section}
\newtheorem{theorem}{Theorem}[section]

\renewcommand{\phi}{\varphi}
\newcommand{\QED}{\hspace*{\fill}\rule{2.5mm}{2.5mm}}

\renewcommand{\vec}[1]{\mbox{\boldmath $#1$}}

\renewcommand{\vec}[1]{\mbox{\boldmath $#1$}}

\allowdisplaybreaks
\begin{document}
\title{Parallelizing the Kolmogorov-Fokker-Planck Equation}
\author{Luca Gerardo-Giorda(1) \\ Minh-Binh Tran(2)\\
 Basque Center for Applied Mathematics\\
 Mazarredo 14, 48009 Bilbao Spain\\
Email: (1) lgerardo@bcamath.org\\
 (2) tbinh@bcamath.org
 }
\maketitle
\begin{abstract}
We design  the first parallel scheme based on Schwarz waveform relaxation methods for the Kolmogorov-Fokker-Planck equation. We  introduce a new convergence proof for the algorithms. We also provide results about the existence and uniqueness of a solution for this equation with several boundary conditions, in order to prove that our algorithms are well-posed.  Numerical tests are also provided. 
\end{abstract}
{\bf Keyword}
{Domain decomposition, Schwarz waveform relaxation methods, optimized Schwarz, Kolmogorov equation, Fokker-Plank equation. 
\section{Introduction}
The Fokker-Planck equation describes the time evolution of the probability density function of the velocity of a particle. It reads for $(t,x,v)\in \mathbb{R}_+^d\times\mathbb{R}^d\times\mathbb{R}^d$, $(d\geq 1)$
\begin{equation}\label{FPE}
\partial_t u+ v\cdot\nabla_x u-\nabla_x V(x)\cdot\nabla_v u=\nabla_v\cdot(\nabla_v u+vu),
\end{equation}
where $V(x)$ is the external potential. Together with the theoretical study of the equation (\cite{DesvillettesVillani:2001:OTG}, \cite{DolbeaultMouhot:2009:HKE}), there are a lot of numerical studies on the Fokker-Plank  and related equations  (\cite{Carrillo:2008:CES}, \cite{Carrillo:2011:ANS},  \cite{BuetDellacherieSentis:2001:NAI}, \cite{LakestaniDehghan:2009:NSFP}, \cite{Dejan:2011:EMN}, \cite{KnezevicSuli:2009:SGA},  \cite{DegondLucquin:1994:AES}), fractional Fokker-Plank equation (\cite{Deng:2009:FEM}),   Wigner-Fokker-Plank equation (\cite{GambaGualdaniSharp:2009:ADG}), Fokker-Planck-Landau equation  \cite{BuetCordierDegondLemou:1997:FAN},  \cite{LemouMieussens:2005:ISF}, \cite{FilbetPareschi:2003:NSFPL}), Vlasov-Fokker-Planck system (\cite{AsadzadehSopasakis:2007:CHP}, \cite{CrouseillesFilbet:2004:NAC}), Vlasov-Poisson-Fokker-Planck system (\cite{Schaeffer:1998:CDS}), Maxwell-Fokker-Planck-Landau equation (\cite{DuclousDubrocaFilbetTikhonchuk:2009:HOR}), Vlasov-Fokker-Planck-Landau equation (\cite{CrouseillesFilbet:2005:CEM}). However, up to our knowledge, there has been no  scheme to parallelize the resolution of these types of kinetic equations.
\\ Parallel computing is a form of computation in which many calculations are carried out in parallel, based on the principle that large problems can be divided into smaller ones. Due to the physical constraints of computers, parallelism has got more and more attention in the recent years. In the last two decades, domain decomposition methods have become a very useful tool to parallelize the numerical resolution of partial differential equations numerically. Schwarz waveform relaxation methods, together with its accelerated version optimized Schwarz waveform relaxation algorithms, is a new class of domain decomposition algorithms adapted to the context of studying evolution equations numerically. For a survey on this, we refer to \cite{Halpern:2009:OSW} and the pioneering works \cite{GanderStuart:1998:STC}, \cite{GanderHalpernNataf:1999:OCO}, \cite{GanderHalpernNataf:1999:OSM}, \cite{GanderHalpern:2001:MDD}, \cite{GanderHalpern:2001:UAD}, \cite{GanderHalpernMagoules:2007:AOS}. 
\\ The main feature of our present work  is to design parallel schemes based on the Schwarz waveform relaxation methods to solve numerically a simplified version of the Fokker-Planck model $(\ref{FPE})$: the Kolmogorov equation  
\begin{equation}\label{Intro1}
\frac{\partial u}{\partial t}+v\frac{\partial u}{\partial x}-\frac{\partial^2 u}{\partial v^2}=f.
\end{equation}
 As we can see from its form, the Kolmogorov equation diffuses not only in the velocity variable, since it contains the diffusion term $\frac{\partial^2 u}{\partial v^2}$, but also in the space variable, because of the hidden interaction between the transport term $  v\frac{\partial u}{\partial x}$ and the diffusion term $\frac{\partial^2 u}{\partial v^2}$. The hypoellipticity and the asymptotic behavior of this operator have been studied in the work of L. Hormander \cite{Hormander:1967:HSO} and of C. Villani \cite{Villani:2006:HDO}. Recently, the null controllability property of this operator has been explored  deeply by K. Beauchard and E. Zuazua in \cite{BeauchardZuazua:2009:SCR}.  
\\ Since the principal part of the operator involves the second derivatives in $v$, we design some Schwarz waveform relaxation algorithms with Dirichlet (classical Schwarz method) or Robin (optimized Schwarz method) transmission condition  for this equation, by splitting the domain in the $v$ direction. For the sake of simplicity, we only split the domain into two subdomains, however, the extension to a larger number of subdomains does not present any theoretical difficulties.
\\ We provide some results on the existence and uniqueness of a solution for the Kolmogorov equation with different boundary conditions, in order to prove that our algorithms are well-posed. 
The convergence proof of Schwarz methods at the continuous level has been a very difficult task. In  \cite{Tran:2012:OOS1}, \cite{Tran:2012:OOS}, \cite{Tran:2010:PSW}, \cite{Tran:2011:PFS}, a new class of techniques has been introduced in order to study this convergence problem of domain decomposition methods. Based on these techniques, we give a new proof of the convergence of our algorithms by some maximum principles and some energy estimates. 
\\ The structure of the paper is the following:
\\ Section $\ref{sec1}$ is devoted to the definition of the equation and the algorithms. In section $\ref{sec2}$, an existence and uniqueness result will be proven for $(\ref{Kolmogorovmain})$, and the subproblems $(\ref{ClassicalSchwarz})$ and $(\ref{RobinSchwarz})$. Since the problems in section $\ref{sec2}$ are considered in the general setting, the domains are chosen in a general manner. In section $\ref{sec3}$ and $\ref{sec4}$ the convergence proofs of the algorithms are given. The numerical experiments are given in Section $\ref{numer}$. We conclude the paper by Section $\ref{secconclu}$.
\section{General Setting}\label{sec1}
We consider the following Kolmogorov equation
\begin{equation}\label{Kolmogorovmain}
\left\{ \begin{array}{ll} \frac{\partial u}{\partial t}+v\frac{\partial u}{\partial x}-\frac{\partial^2 u}{\partial v^2}=f, \mbox{ in } (0,\infty)\times\Omega:=(0,\infty)\times\mathbb{T}\times\mathbb{R}_v,\vspace{.1in}\\
u(0,x,v)=u_0(x,v), \mbox{ on } \mathbb{T}\times\mathbb{R}_v,\end{array}\right.
\end{equation}
where $\mathbb{T}$ is the periodic domain $\mathbb{R}\slash\mathbb{Z}$, $f\in L^\infty(0,\infty,L^2(\mathbb{T},H^1(\mathbb{R}_v)))\cap L^\infty((0,\infty)\times\mathbb{T}\times\mathbb{R}_v)$, $u_0\in L^2(\mathbb{T},H^1(\mathbb{R}_v))\cap L^\infty(\mathbb{T}\times\mathbb{R}_v)$.
\\ {\it Parallel domain decomposition algorithms} consist of dividing the domain $\Omega$ into two parts $\Omega_1:=\mathbb{T}\times(-\infty,L_2)$ and $\Omega_2:=\mathbb{T}\times(L_1,\infty)$, where $L_1<L_2$, and solving $(\ref{Kolmogorovmain})$ parallely on each subdomain $\Omega_1$ and $\Omega_2$. 
\\The {\it classical Schwarz waveform relaxation algorithm} for $(\ref{Kolmogorovmain})$ is then written
\begin{equation}\label{ClassicalSchwarz}
\left\{ \begin{array}{ll} \frac{\partial u_1^n}{\partial t}+v\frac{\partial u_1^n}{\partial x}-\frac{\partial^2 u_1^n}{\partial v^2}=f, \mbox{ in } (0,\infty)\times\Omega_1,\vspace{.1in}\\
u_1^n(0,x,v)=u_0(x,v), \mbox{ on } \Omega_1,\vspace{.1in}\\ 
u_1^n(t,x,L_1)=u_2^{n-1}(t,x,L_1),\mbox{ on } (0,\infty)\times\mathbb{T},\end{array}\right.
\end{equation}
and
\begin{eqnarray*}
\left\{ \begin{array}{ll} \frac{\partial u_2^n}{\partial t}+v\frac{\partial u_2^n}{\partial x}-\frac{\partial^2 u_2^n}{\partial v^2}=f, \mbox{ in } (0,\infty)\times\Omega_2,\vspace{.1in}\\
u_2^n(0,x,v)=u_0(x,v), \mbox{ on } \Omega_2,\vspace{.1in}\\ 
u_2^n(t,x,L_2)=u_1^{n-1}(t,x,L_2),\mbox{ on } (0,\infty)\times\mathbb{T},\end{array}\right.
\end{eqnarray*}
the initial guess $u_1^0$ and $u_2^0$ are chosen arbitrarily in $L^\infty(0,\infty,L^\infty(\mathbb{T}))$ and satisfy the compatibility conditions of the equations:
\begin{eqnarray*}
\left\{ \begin{array}{ll} u_1^0(0,x)=u_0(x,L_2),\mbox{ on } \mathbb{T}\vspace{.1in}\\ 
 u_2^0(0,x)=u_0(x,L_1),\mbox{ on } \mathbb{T}\vspace{.1in}.\end{array}\right.
\end{eqnarray*}
When $n$ tends to $\infty$, $u_1^n$ and $u_2^n$ are expected to converge to $u$ on $\Omega_1$ and $\Omega_2$.
\\  Let $p,q$ be two positive numbers, the {\it optimized Schwarz waveform relaxation algorithm} for $(\ref{Kolmogorovmain})$ is defined by replacing the Dirichlet transmission condition in $(\ref{ClassicalSchwarz})$ 
\begin{equation}\label{RobinSchwarz}
\left\{ \begin{array}{ll} \frac{\partial u_1^n}{\partial t}+v\frac{\partial u_1^n}{\partial x}-\frac{\partial^2 u_1^n}{\partial v^2}=f, \mbox{ in } (0,\infty)\times\Omega_1,\vspace{.1in}\\
u_1^n(0,x,v)=u_0(x,v), \mbox{ on } \Omega_1,\vspace{.1in}\\ 
(p+\frac{\partial}{\partial v})u_1^n(t,x,L_1)=(p+\frac{\partial}{\partial v})u_2^{n-1}(t,x,L_1),\mbox{ on } (0,\infty)\times\mathbb{T},\end{array}\right.
\end{equation}
and
\begin{eqnarray*}
\left\{ \begin{array}{ll} \frac{\partial u_2^n}{\partial t}+v\frac{\partial u_2^n}{\partial x}-\frac{\partial^2 u_2^n}{\partial v^2}=f, \mbox{ in } (0,\infty)\times\Omega_2,\vspace{.1in}\\
u_2^n(0,x,v)=u_0(x,v), \mbox{ on } \Omega_2,\vspace{.1in}\\ 
(q-\frac{\partial}{\partial v})u_2^n(t,x,L_2)=(q-\frac{\partial}{\partial v})u_1^{n-1}(t,x,L_2),\mbox{ on } (0,\infty)\times\mathbb{T},\end{array}\right.
\end{eqnarray*}
where at the first iteration the initial guesses $u_1^0$, $u_1^0$ are chosen such that  $(p+\frac{\partial}{\partial v})u_1^0(t,x,L_2)$ and $(q-\frac{\partial}{\partial v})u_2^0(t,x,L_1)$ are  in $L^\infty(0,\infty,L^\infty(\mathbb{T}))$ and satisfy the compatibility conditions of the equations:
\begin{eqnarray*}
\left\{ \begin{array}{ll} u_1^0(0,x)=u_0(x,L_2),\mbox{ on } \mathbb{T}\vspace{.1in}\\ 
 u_2^0(0,x)=u_0(x,L_1),\mbox{ on } \mathbb{T}\vspace{.1in}.\end{array}\right.
\end{eqnarray*}
Compared with the classical Schwarz waveform relaxation algorithm, optimized ones make $u_1^n$ and $u_2^n$ converge to $u$ in less iterations. Moreover, optimized Schwarz algorithms converge also in the non-overlapping case, a feature not shared by the classical ones.
\section{Existence and Uniqueness Results for the Kolomogorov Equations}\label{sec2}
In this section, we will prove the existence and uniqueness of a solution of the Kolmogorov equation
\begin{equation}\label{Kolmogorov}
\left \{ \begin{array}{ll}\frac{\partial u}{\partial t}+v\frac{\partial u}{\partial x}-\frac{\partial^2 u}{\partial v^2}=f \mbox{ for } (t,x,v)\in(0,\infty)\times\mathbb{T}\times(a,b)\subset\Omega,\vspace{.1in}\\ 
u(0,x,v)=u_0 \mbox{ in } \mathbb{T}\times(a,b).\end{array}\right. 
\end{equation}
where $(a,b)$ could be $\mathbb{R}$, $(-\infty,L_2)$ or $(L_1,\infty)$, $f\in L^\infty((0,\infty),L^2(\mathbb{T}\times(a,b)))$, $u_0\in L^2(\mathbb{T}\times(a,b))$.\\
The boundary conditions that we consider here are of the following two types: 
\begin{itemize}
\item Dirichlet boundary condition 
\begin{equation}\label{DCon1}
u(t,x,L_2)=h_0(t,x),\mbox{ on } (0,\infty)\times\mathbb{T},
\end{equation}
and
\begin{equation}\label{DCon2}
u(t,x,L_1)=h_0(t,x),\mbox{ on } (0,\infty)\times\mathbb{T}, 
\end{equation}
\item Robin boundary condition
\begin{equation}\label{RCon3}
pu(t,x,L_2)+\frac{\partial u(t,x,L_2)}{\partial v}=h_1(t,x),\mbox{ on } (0,\infty)\times\mathbb{T}, 
\end{equation}
and
\begin{equation}\label{RCon4}
qu(t,x,L_1)-\frac{\partial u(t,x,L_1)}{\partial v}=h_1(t,x),\mbox{ on } (0,\infty)\times\mathbb{T}, 
\end{equation}
for $h_0$, $h_1$ $\in L^\infty((0,\infty),L^2(\mathbb{T}))$.
\end{itemize}
Since there exist functions $\tilde{u}_1$, $\tilde{u}_2$, $\tilde{u}_3$ and $\tilde{u}_4$ in $L^\infty((0,\infty),L^2(\mathbb{T},H^1(a,b)))$  such that
\begin{eqnarray*}\label{DBEquation1}
\tilde{u}_1(t,x,L_2)=h_0(t,x),\mbox{ on } (0,\infty)\times\mathbb{T},
\end{eqnarray*} 
\begin{eqnarray*}\label{DBEquation2}
\tilde{u}_2(t,x,L_1)=h_0(t,x),\mbox{ on } (0,\infty)\times\mathbb{T},
\end{eqnarray*} 
and
\begin{eqnarray*}\label{DBEquation3}
p\tilde{u}_3(t,x,L_2)+\frac{\partial \tilde{u}_3(t,x,L_2)}{\partial v}=h_1(t,x),\mbox{ on } (0,\infty)\times\mathbb{T},
\end{eqnarray*}
\begin{eqnarray*}\label{DBEquation3}
q\tilde{u}_4(t,x,L_2)-\frac{\partial \tilde{u}_4(t,x,L_1)}{\partial v}=h_1(t,x),\mbox{ on } (0,\infty)\times\mathbb{T},
\end{eqnarray*}
then by subtracting $u$ with $\tilde{u}_1$, $\tilde{u}_2$, $\tilde{u}_3$ or $\tilde{u}_4$, we can suppose that $h_0=h_1=0$. 
\\ Take the Fourier transform in $x$ of $(\ref{Kolmogorov})$, 
\begin{equation}\label{Kolmogorov0}
\frac{\partial \hat{u}}{\partial t}+iv\zeta{\hat{u}}-\frac{\partial^2\hat{u}}{\partial v^2}=\hat{f}, \mbox{ on } (0,\infty)\times\mathbb{R}\times(a,b).
\end{equation}
Split $\hat{u}$ and $\hat{f}$ into their real and imaginary parts
$$\hat{u}=\hat{u}_1+i\hat{u}_2,$$
$$\hat{f}=\hat{f}_1+i\hat{f}_2,$$
Equation $(\ref{Kolmogorov0})$ becomes
\begin{eqnarray}\label{Kolmogorov1}
\left \{ \begin{array}{ll}\frac{\partial \hat{u}_1(\zeta)}{\partial t}-v\zeta\hat{u}_2(\zeta)-\frac{\partial^2\hat{u}_1(\zeta)}{\partial v^2}=\hat{f}_1(\zeta), \mbox{ on } (0,\infty)\times(a,b),\vspace{.1in}\\
\frac{\partial \hat{u}_2(\zeta)}{\partial t}+v\zeta{\hat{u}_1}(\zeta)-\frac{\partial^2\hat{u}_2(\zeta)}{\partial v^2}=\hat{f}_2(\zeta), \mbox{ on } (0,\infty)\times(a,b),\end{array}\right. 
\end{eqnarray}
the four boundary conditions remains the same after this transformation.
\begin{equation}\label{DCon1b}
\hat{u}(t,\zeta,L_2)=0,\mbox{ on } (0,\infty)\times\mathbb{R},
\end{equation}
\begin{equation}\label{DCon2b}
\hat{u}(t,\zeta,L_1)=0,\mbox{ on } (0,\infty)\times\mathbb{R}, 
\end{equation}
and
\begin{equation}\label{RCon3b}
p\hat{u}(t,\zeta,L_2)+\frac{\partial \hat{u}(t,\zeta,L_2)}{\partial v}=0,\mbox{ on } (0,\infty)\times\mathbb{R}, 
\end{equation}
\begin{equation}\label{RCon4b}
q\hat{u}(t,\zeta,L_1)-\frac{\partial \hat{u}(t,\zeta,L_1)}{\partial v}=0,\mbox{ on } (0,\infty)\times\mathbb{R}, 
\end{equation} For any given $\zeta$, since $\hat{f}_1(\zeta)$, $\hat{f}_2(\zeta)\in L^\infty((0,\infty),L^2(a,b))$, there exists a solution $(\hat{u}_1(\zeta),\hat{u}_2(\zeta))$ in  $L_{loc}^\infty((0,\infty),H^1(a,b))$ of $(\ref{Kolmogorov1})$ (see, for example \cite[Chapter VII]{LS:1967:LQE}). \\
Choose $\zeta$ to be an integer $n$ and use $\hat{u}_1(n)$ and $\hat{u}_2(n)$ as test functions for the system $(\ref{Kolmogorov1})$, 
\begin{eqnarray*}
\left \{ \begin{array}{ll}\frac{1}{2}\int_a^b\frac{\partial |\hat{u}_1(n)|^2}{\partial t}dv-\int_a^b vn\hat{u}_2(n)\hat{u}_1(n)dv-\int_a^b \frac{\partial^2\hat{u}_1(n)}{\partial v^2}\hat{u}_1(n)dv=\int_a^b\hat{f}_1(n)\hat{u}_1(n)dv ,\vspace{.1in}\\
\frac{1}{2}\int_a^b\frac{\partial |\hat{u}_2(n)|^2}{\partial t}dv+\int_a^b vn\hat{u}_1(n)\hat{u}_2(n)dv-\int_a^b \frac{\partial^2\hat{u}_2(n)}{\partial v^2}\hat{u}_2(n)dv=\int_a^b\hat{f}_2(n)\hat{u}_2(n)dv.\end{array}\right. 
\end{eqnarray*}
Adding the two equations, using one of the boundary conditions $(\ref{DCon1b})$, $(\ref{DCon2b})$, $(\ref{RCon3b})$, $(\ref{RCon4b})$ for the case $(a,b)\ne \mathbb{R}$, and taking into account the fact that $p$ and $q$ are positive, we get
\begin{eqnarray*}
& &\frac{1}{2}\int_a^b\frac{\partial |\hat{u}_1(n)|^2}{\partial t}dv+\frac{1}{2}\int_a^b\frac{\partial |\hat{u}_2(n)|^2}{\partial t}dv+\int_a^b\left|\frac{\partial \hat{u}_1(n)}{\partial v}\right|^2dv+\int_a^b\left|\frac{\partial \hat{u}_2(n)}{\partial v}\right|^2dv\vspace{.1in}\\\nonumber 
&\leq &\int_a^b\hat{f}_1(n)\hat{u}_1(n)dv+\int_a^b\hat{f}_2(n)\hat{u}_2(n)\\\nonumber
&\leq &\frac{1}{2}\int_a^b|\hat{f}_1(n)|^2dv+\frac{1}{2}\int_a^b|\hat{u}_1(n)|^2dv+\frac{1}{2}\int_a^b|\hat{f}_2(n)|^2dv+\frac{1}{2}\int_a^b|\hat{u}_2(n)|^2dv,
\end{eqnarray*}
then
\begin{eqnarray*}
& &(\int_a^b\frac{\partial |\hat{u}_1(n)|^2}{\partial t}dv+\int_a^b\frac{\partial |\hat{u}_2(n)|^2}{\partial t}dv)-(\int_a^b|\hat{u}_1(n)|^2dv+\int_a^b|\hat{u}_2(n)|^2dv)\\\nonumber
&\leq &\int_a^b|\hat{f}_1(n)|^2dv+\int_a^b|\hat{f}_2(n)|^2dv.
\end{eqnarray*}
The previous inequality implies
\begin{eqnarray*}
& &\partial_t\left({(\int_a^b|\hat{u}_1(n)|^2dv+\int_a^b|\hat{u}_2(n)|^2dv)\exp(-t)}\right)\\\nonumber
&\leq &\left(\int_a^b|\hat{f}_1(n)|^2dv+\int_a^b|\hat{f}_2(n)|^2dv\right)\exp(-t).
\end{eqnarray*}
Thus
\begin{eqnarray*}
& &\int_a^b(|\hat{u}_1(n,t)|^2+|\hat{u}_2(n,t)|^2)dv\\
&\leq& \int_0^t\exp(s-t)ds\int_a^b(\|\hat{f}_1(n)\|_{L^\infty(0,\infty)}^2+\|\hat{f}_2(n)\|_{L^\infty(0,\infty)}^2)dv+\int_a^b(|\hat{u}_1(n,0)|^2+|\hat{u}_2(n,0)|^2)dv\\
&\leq&C\exp(t) \int_a^b(\|\hat{f}_1(n)\|_{L^\infty(0,\infty)}^2+\|\hat{f}_2(n)\|_{L^\infty(0,\infty)}^2)dv+\int_a^b(|\hat{u}_1(n,0)|^2+|\hat{u}_2(n,0)|^2)dv,
\end{eqnarray*}
Summing up in $\mathbb{Z}$ the previous inequalities yields
\begin{eqnarray*}
& &\int_a^b\sum_{n\in\mathbb{Z}}(|\hat{u}_1(n,t)|^2+|\hat{u}_2(n,t)|^2)dv\\
&\leq& C\exp(t)\int_a^b\sum_{n\in\mathbb{Z}}(\|\hat{f}_1(n)\|_{L^\infty(0,\infty)}^2+\|\hat{f}_2(n)\|_{L^\infty(0,\infty)}^2)dv+\int_a^b\sum_{n\in\mathbb{Z}}(|\hat{u}_1(n,0)|^2+|\hat{u}_2(n,0)|^2)dv,
\end{eqnarray*}
which together with the Parseval's theorem implies 
\begin{eqnarray*}
\int_a^b\int_{\mathbb{R}}|\hat{u}(\zeta,t)|^2d\zeta dv\leq C\exp(t)\int_a^b\int_{\mathbb{R}}\|\hat{f}(\zeta)\|_{L^\infty(0,\infty)}^2d\zeta dv+\int_a^b\int_{\mathbb{R}}|\hat{u}(\zeta,0)|^2d\zeta dv.
\end{eqnarray*}
 Therefore, the inverse Fourier transform $u$ of $\hat{u}_1$ and $\hat{u}_2$ exists and
\begin{eqnarray}\label{Laplacewelldefined}
\int_{\mathbb{T}}\int_a^b|u(t)|^2dvdx\leq C\exp(t)\int_{\mathbb{T}}\int_a^b\|f\|_{L^\infty(0,\infty)}^2dvdx+\int_{\mathbb{T}}\int_a^b|u(0)|^2dvdx.
\end{eqnarray} 
The existence and uniqueness of a solution of $(\ref{Kolmogorov})$ with one of the above boundary conditions then follows. 
\begin{theorem}\label{RExistence}
Suppose that $h_0$, $h_1$ $\in L^\infty((0,\infty),L^2(\mathbb{T}))$, $f\in L^\infty((0,\infty),L^2(\mathbb{T}\times(a,b)))$, $u_0\in  L^2(\mathbb{T}\times(a,b))$, Equation $(\ref{Kolmogorov})$, with either one of the boundary conditions $(\ref{DCon1})$, $(\ref{DCon2})$,  $(\ref{RCon3})$ , $(\ref{RCon4})$ or without boundary condition in the case $(a,b)=\mathbb{R}_v$ has a unique solution in $ L_{loc}^\infty(0,\infty,L^2(\mathbb{T},H^2(a,b)))$.
\end{theorem}
By a classical induction argument as in \cite{GanderHalpern:2007:OSW}, we have also the well-posedness of the algorithm. 
\begin{theorem}\label{Wellposed}
Suppose that $f\in $ $L^\infty((0,T),L^2(\mathbb{T}\times(a,b)))$, $u_0\in  L^2(\mathbb{T}\times\mathbb{R}_v)$ and  the initial guesses for the Dirichlet transmission condition $u_0^1$, $u_0^2$ $\in L^\infty((0,\infty),L^2(\mathbb{T}))$, the initial guesses for the Robin transmission condition $u_1^0$, $u_1^0$ are chosen such that  $(p+\frac{\partial}{\partial v})u_1^0(t,x,L_2)$ and $(q-\frac{\partial}{\partial v})u_2^0(t,x,L_1)$ are  in $L^\infty(0,\infty,L^2(\mathbb{T}))$, Equations $(\ref{ClassicalSchwarz})$ and $(\ref{RobinSchwarz})$ have unique solutions in   $ L_{loc}^\infty(0,\infty,L^2(\mathbb{T},H^2(-\infty,L_2)))$ and $ L_{loc}^\infty(0,\infty,L^2(\mathbb{T},H^2(L_1,\infty)))$.
\end{theorem}
\section{Convergence of the Classical Schwarz Waveform Relaxation Algorithm}\label{sec3}
\begin{theorem}\label{ConvergenceClassical}
Suppose that $L_1<L_2$. For all positive number $T$, the algorithm converges in the following sense 
$$\lim_{n\to\infty}\|u_1^n-u\|_{L^\infty((0,T)\times\Omega_1)}=0,$$
and
$$\lim_{n\to\infty}\|u_2^n-u\|_{L^\infty((0,T)\times\Omega_2)}=0.$$
\end{theorem}
Since the problems are linear in $u$, we can prove the convergence on the error equation by letting $e_1^n=u_1^n-u$ and $e_2^n=u_2^n-u$, then
\begin{equation}\label{ClassicalSchwarz}
\left\{ \begin{array}{ll} \frac{\partial e_1^n}{\partial t}+v\frac{\partial e_1^n}{\partial x}-\frac{\partial^2 e_1^n}{\partial v^2}=0, \mbox{ in } (0,\infty)\times\Omega_1,\vspace{.1in}\\
e_1^n(0,x,v)=0, \mbox{ on } \Omega_1,\vspace{.1in}\\ 
e_1^n(t,x,L_2)=e_2^{n-1}(t,x,L_2),\mbox{ on } (0,\infty)\times\mathbb{T},\end{array}\right.
\end{equation}
\begin{eqnarray*}
\left\{ \begin{array}{ll} \frac{\partial e_2^n}{\partial t}+v\frac{\partial e_2^n}{\partial x}-\frac{\partial^2 e_2^n}{\partial v^2}=0, \mbox{ in } (0,\infty)\times\Omega_2,\vspace{.1in}\\
e_2^n(0,x,v)=0, \mbox{ on } \Omega_2,\vspace{.1in}\\ 
e_2^n(t,x,L_1)=e_1^{n-1}(t,x,L_1),\mbox{ on } (0,\infty)\times\mathbb{T}.\end{array}\right.
\end{eqnarray*}
\begin{proof}
Let $f_1$ and $f_2$ be two strictly positive and continuous functions on $\mathbb{R}$. Define
$$\Phi_1^n=(e_1^n)^2f_1(t)f_2(v),$$
$$\Phi_2^n=(e_2^n)^2f_1(t)f_2(v).$$
A simple calculation leads to
\begin{equation}\label{DEqPhiLeftf}
\frac{\partial}{\partial t}\Phi_1^n-\frac{\partial^2}{\partial v^2}\Phi_1^n+v\frac{\partial}{\partial x}\Phi_1^n+2\frac{f_2'}{f_2}\frac{\partial}{\partial v}\Phi_1^n+\left(-\frac{f_1'}{f_1}+\frac{f_2''}{f_2}-2\left(\frac{f_2'}{f_2}\right)^2\right)\Phi_1^n\leq 0,
\end{equation}
and
\begin{equation}\label{DEqPhiRightf}
\frac{\partial}{\partial t}\Phi_2^n-\frac{\partial^2}{\partial v^2}\Phi_2^n+v\frac{\partial}{\partial x}\Phi_2^n+2\frac{f_2'}{f_2}\frac{\partial}{\partial v}\Phi_2^n+\left(-\frac{f_1'}{f_1}+\frac{f_2''}{f_2}-2\left(\frac{f_2'}{f_2}\right)^2\right)\Phi_2^n\leq 0.
\end{equation}
Let $\alpha$ be a constant to be chosen later and put $f_1(t)=\exp(-\alpha^2 t)$, $f_2(v)=\exp(\alpha v)$ to get
\begin{equation}\label{DEqPhiLeft}
\frac{\partial}{\partial t}\Phi_1^n-\frac{\partial^2}{\partial v^2}\Phi_1^n+v\frac{\partial}{\partial x}\Phi_1^n+2\alpha\frac{\partial}{\partial v}\Phi_1^n\leq 0,
\end{equation}
and
\begin{equation}\label{DEqPhiRight}
\frac{\partial}{\partial t}\Phi_2^n-\frac{\partial^2}{\partial v^2}\Phi_2^n+v\frac{\partial}{\partial x}\Phi_2^n+2\alpha\frac{\partial}{\partial v}\Phi_2^n\leq 0.
\end{equation}
{\bf Step 1:} The maximum principle.
\\ We prove that the solution $u$ of $(\ref{Kolmogorovmain})$ belongs to  $L^\infty([0,T]\times\mathbb{T}\times\mathbb{R}_v)$. Let $K$ be greater than $\|f\|_{L^\infty([0,T]\times\mathbb{T}\times\mathbb{R}_v)}$ and $\|u_0\|_{L^\infty(\mathbb{T}\times\mathbb{R}_v)}$, then 
\begin{eqnarray*}
\left\{ \begin{array}{ll} \frac{\partial (u-Kt)}{\partial t}+v\frac{\partial (u-Kt)}{\partial x}-\frac{\partial^2 (u-Kt)}{\partial v^2}=f-K, \mbox{ in } (0,T)\times\Omega,\vspace{.1in}\\
(u-Kt)(0,x,v)=u_0(x,v), \mbox{ on } \mathbb{T}\times\mathbb{R}_v.\end{array}\right.
\end{eqnarray*}
Set $\psi=(u-Kt)_+$, where $\psi=(u-Kt)$ for $u\geq Kt$ and $0$ for $u<Kt$. Using $\psi$ as a test function for the above equation, we get
\begin{eqnarray*}
0&\geq&\int_0^T\int_{\mathbb{R}_v}\int_{\mathbb{T}}\frac{\partial}{\partial t}(u-Kt)(u-Kt)_+dxdvdt\\
& &+\int_0^T\int_{\mathbb{R}_v}\int_{\mathbb{T}}v\frac{\partial}{\partial x}(u-Kt)(u-Kt)_+dxdvdt\\
& &+\int_0^T\int_{\mathbb{R}_v}\int_{\mathbb{T}}\frac{\partial}{\partial v}(u-Kt)\frac{\partial}{\partial v}(u-Kt)_+dxdvdt,
\end{eqnarray*}
which yields
\begin{eqnarray*}
0&\geq&\int_{\mathbb{R}_v}\int_{\mathbb{T}}\frac{(u-Kt)_+^2}{2}|_0^Tdxdv+\int_0^T\int_{\mathbb{R}_v}\int_{\mathbb{T}}\left(\frac{\partial}{\partial v}(u-K)_+\right)^2dxdvdt.
\end{eqnarray*}
Hence $(u-Kt)_+=0$, then $u\leq KT$ or $u$ is bounded from above. By a similar argument, we can prove that $u$ is bounded equivalently from below, and $u\in L^\infty([0,T]\times\mathbb{T}\times\mathbb{R}_v)$. 
\\ Let $M=\sup_{(t,x)\in(0,T)\times\mathbb{T}}\{\Phi_2^{n-1}(t,x,L_2)\}$ and suppose that $M<\infty$. Notice that $u_2^0\in L^\infty([0,T]\times\mathbb{T}\times[L_1,\infty))$ and $u\in L^\infty([0,T]\times\mathbb{T}\times\mathbb{R}_v)$, then $M<\infty$ for $n=1$. Set $\varphi=(\Phi_1^n-M)_+$, where $\varphi=\Phi_1^n-M$ for $\Phi_1^n\geq M$ and $0$ for $\Phi_1^n<M$. Using it as a test function for $(\ref{DEqPhiLeft})$, we obtain
\begin{eqnarray*}
0&=&\int_0^T\int_{-\infty}^{L_2}\int_{\mathbb{T}}\frac{\partial}{\partial t}(\Phi_1^n-M)(\Phi_1^n-M)_+dxdvdt\\
& &+\int_0^T\int_{-\infty}^{L_2}\int_{\mathbb{T}}v\frac{\partial}{\partial x}(\Phi_1^n-M)(\Phi_1^n-M)_+dxdvdt\\
& &+\int_0^T\int_{-\infty}^{L_2}\int_{\mathbb{T}}\frac{\partial}{\partial v}(\Phi_1^n-M)\frac{\partial}{\partial v}(\Phi_1^n-M)_+dxdvdt\\
& &+2\alpha\int_0^T\int_{-\infty}^{L_2}\int_{\mathbb{T}}\frac{\partial}{\partial v}(\Phi_1^n-M)(\Phi_1^n-M)_+dxdvdt.
\end{eqnarray*}
This leads to 
\begin{eqnarray*}
0&=&\int_{-\infty}^{L_2}\int_{\mathbb{T}}\frac{(\Phi_1^n-M)_+^2}{2}|_0^Tdxdv+\int_0^T\int_{-\infty}^{L_2}\int_{\mathbb{T}}\left(\frac{\partial}{\partial v}(\Phi_1^n-M)_+\right)^2dxdvdt\\
& &+2\alpha\int_0^T\int_{\mathbb{T}}(\Phi_1^n-M)_+^2|_0^{L_2}dxdt,
\end{eqnarray*}
which gives $(\Phi_1^n-M)_+=0$. As a consequence,
$$\Phi_1^n\leq M,$$
or 
\begin{equation}\label{maximumprinciple1}
\Phi_1^n(t,x,v)\leq \sup_{(t,x)\in(0,T)\times\mathbb{T}}\{\Phi_2^{n-1}(t,x,L_2)\} \mbox{ on }\Omega_1.
\end{equation}
A similar argument leads to
\begin{equation}\label{maximumprinciple2}
\Phi_2^n(t,x,v)\leq \sup_{(t,x)\in(0,T)\times\mathbb{T}}\{\Phi_1^{n-1}(t,x,L_1)\} \mbox{ on }\Omega_2.
\end{equation}
{\bf Step 2:} The convergence estimates.\\
Denote
\begin{eqnarray*}
E^n=\max_{i\in \{1,2\}}\left(\sup_{(t,x)\in((0,T)\times\Omega_i)}\left(e_i^n\right)^2 f_1(t)\right).
\end{eqnarray*}
Since $e_1^0$ and $e_2^0$ are bounded, $E^n$ is bounded.
\\ Inequality $(\ref{maximumprinciple1})$ implies that for $(x,t)$ in $(0,T)\times\mathbb{T}$
\begin{eqnarray*}
\left(e_{1}^n(t,x,L_1)\right)^2f_1(t)f_2(L_1)\leq \sup_{(t,x)\in(0,T)\times\mathbb{T}}(e_{2}^{n-1}(t,x,L_2))^2f_1(t)f_2(L_2),
\end{eqnarray*}
which yields
\begin{eqnarray*}
\left(e_{1}^n(t,x,L_1)\right)^2f_1(t)\leq \exp((L_2-L_1)\alpha)\sup_{(t,x)\in(0,T)\times\mathbb{T}}(e_{2}^{n-1}(t,x,L_2))^2f_1(t).
\end{eqnarray*}
Choosing $\alpha=-\alpha_0$ where $\alpha_0$ is a positive constant to get
\begin{eqnarray*}
\left(e_{1}^n(t,x,L_1)\right)^2f_1(t)\leq \exp((L_1-L_2)\alpha_0)\sup_{(t,x)\in(0,T)\times\mathbb{T}}(e_{2}^{n-1}(t,x,L_2))^2f_1(t).
\end{eqnarray*}
Similarly, by using the same argument and replacing $\alpha$ by $\alpha_0$
\begin{eqnarray*}
\left(e_{2}^n(t,x,L_2)\right)^2f_1(t)\leq \exp((L_1-L_2)\alpha_0)\sup_{(t,x)\in(0,T)\times\mathbb{T}}(e_{1}^{n-1}(t,x,L_1))^2f_1(t).
\end{eqnarray*}
Choose $\alpha=0$, $(\ref{maximumprinciple1})$ and $(\ref{maximumprinciple2})$ imply
\begin{eqnarray*}
E^{n+1}\leq\max\{\sup_{(t,x)\in(0,T)\times\mathbb{T}}\left(e_{1}^n(t,x,L_2)\right)^2f_1(t),\sup_{(t,x)\in(0,T)\times\mathbb{T}}\left(e_{2}^n(t,x,L_1)\right)^2f_1(t)\}.
\end{eqnarray*}
The above inequality implies
\begin{eqnarray*}
E^{n+1}& \leq &\exp((L_1-L_2)\alpha_0)\max\{\sup_{(t,x)\in(0,T)\times\mathbb{T}}(e_{2}^{n-1}(t,x,L_2))^2f_1(t),\\ & &\sup_{(t,x)\in(0,T)\times\mathbb{T}}(e_{2}^{n-1}(t,x,L_2))^2f_1(t)\}.\\
&\leq &\exp((L_1-L_2)\alpha_0) E^{n-1},
\end{eqnarray*}
which shows that the errors converge geometrically $$\lim_{k\to\infty}E^k=0.$$ 
\end{proof}
\section{Convergence of the Schwarz Waveform Relaxation Methods with Robin Transmission Conditions}\label{sec4}
Again, we prove the convergence on the error equation by letting $e_1^n=u_1^n-u$ and $e_2^n=u_2^n-u$, we consider
\begin{equation}\label{RobinSchwarz}
\left\{ \begin{array}{ll} \frac{\partial e_1^n}{\partial t}+v\frac{\partial e_1^n}{\partial x}-\frac{\partial^2 e_1^n}{\partial v^2}=0, \mbox{ in } (0,\infty)\times\Omega_1,\vspace{.1in}\\
e_1^n(0,x,v)=0, \mbox{ on } \Omega_1,\vspace{.1in}\\ 
(p+\frac{\partial}{\partial v})e_1^n(t,x,L_1)=(p+\frac{\partial}{\partial v})e_2^{n-1}(t,x,L_1),\mbox{ on } (0,\infty)\times\mathbb{T},,\end{array}\right.
\end{equation}
\begin{eqnarray*}
\left\{ \begin{array}{ll} \frac{\partial e_2^n}{\partial t}+v\frac{\partial e_2^n}{\partial x}-\frac{\partial^2 e_2^n}{\partial v^2}=0, \mbox{ in } (0,\infty)\times\Omega_2,\vspace{.1in}\\
e_2^n(0,x,v)=0, \mbox{ on } \Omega_2,\vspace{.1in}\\ 
(q-\frac{\partial}{\partial v})e_2^n(t,x,L_2)=(q-\frac{\partial}{\partial v})e_1^{n-1}(t,x,L_2),\mbox{ on } (0,\infty)\times\mathbb{T}.\end{array}\right.
\end{eqnarray*}
For $\phi$ in $L^2((0,\infty)\times\mathbb{T}\times(a,b))$, where $(a,b)\subset \mathbb{R}_v$, we define the norm
\begin{equation*}
|||\phi|||_{\alpha}=\sup_{\alpha'>\alpha}\left[\int_{\alpha'}^{\alpha'+1}\left(\int_0^{\infty}\phi(x)\exp(-y x)dx\right)^2dy\right]^{\frac{1}{2}},
\end{equation*}
and the space
$$\mathbb{L}_{\alpha}^{2}(0,\infty)=\{\phi~:~\phi\in L^2(0,\infty), |\phi|_{\alpha}<\infty\}.$$
\begin{theorem}\label{ConvergenceRobin}
There exists a positive constant $\alpha$ such that the algorithm converges in the following sense
$$\lim_{n\to\infty}|||u_1^n-u|||_{\alpha}=0,$$
and
$$\lim_{n\to\infty}|||u_2^n-u|||_{\alpha}=0.$$
\end{theorem}
\begin{proof}Let $\alpha$ is a constant and define  $$f_1=\exp(-2\alpha t).$$  
Let $f_2$, $f_3$, $f_4$ be strictly positive functions on $C^\infty(\mathbb{R})$ such that $f_2, f_3, f_4>\beta$, where $\beta$ is some positive constant and $f_2$ is periodic on $\mathbb{T}$. Suppose that $f_2$, $f_3$, $f_4$ satisfy the following assumptions
\begin{eqnarray}\label{assumptionf3f4}
& &\left|\frac{f_2'}{f_2}\right|<\alpha^{1/3},\left|\frac{f_3'}{f_3}\right|^2<\alpha^{1/3},\left|\frac{f_3''}{f_3}\right|<\alpha^{1/3},\left|\frac{f_4'}{f_4}\right|^2<\alpha^{1/3},\left|\frac{f_4''}{f_4}\right|<\alpha^{1/3},\\\nonumber
& &f'_3(L_2)=f'_4(L_1)=0, \\\nonumber
& &\frac{f_4(L_1)}{f_3(L_1)}C_*<\frac{1}{8},\frac{f_3(L_2)}{f_4(L_2)}C_*<\frac{1}{8},
\end{eqnarray}
where $C^*$ is the constant in $(\ref{Rbalance2})$.
\\ Define
$$\Phi_1^{n+1}=\int_0^\infty\int_{\mathbb{T}}e_1^{n+1}f_1f_2f_3dxdt,$$
$$\Phi_2^{n+1}=\int_0^\infty\int_{\mathbb{T}}e_2^{n+1}f_1f_2f_4dxdt.$$
Suppose that $\alpha>1$, using $(\ref{Laplacewelldefined})$, we can see that $\Phi_1^0$ and $\Phi_2^0$ belong to $H^2(-\infty,L_2)$ and $H^2(L_1,\infty)$ for   $n=0$.  
\\ A simple calculation leads to 
\begin{equation}\label{REqSubdomain1}
\int_0^\infty\int_{\mathbb{T}}e_1^{n+1}f_1f_2f_3\left(2\alpha-\frac{f_2'}{f_2}-2\left(\frac{f_3'}{f_3}\right)^2+\frac{f_3''}{f_3}\right)dxdt-\frac{\partial^2}{\partial v^2}\Phi_1^{n+1}+2\frac{f_3'}{f_3}\frac{\partial}{\partial v}\Phi_1^{n+1}=0,
\end{equation}
and
\begin{equation}\label{REqSubdomain2}
\int_0^\infty\int_{\mathbb{T}}e_2^{n+1}f_1f_2f_3\left(2\alpha-\frac{f_2'}{f_2}-2\left(\frac{f_4'}{f_4}\right)^2+\frac{f_4''}{f_4}\right)dxdt-\frac{\partial^2}{\partial v^2}\Phi_2^{n+1}+2\frac{f_4'}{f_4}\frac{\partial}{\partial v}\Phi_2^{n+1}=0.
\end{equation}
The Robin boundary conditions become
\begin{eqnarray}\label{RRobinconditions1}\nonumber
\left(p+\frac{\partial}{\partial v}\right)\Phi_1^{n+1}(x,L_2)&=&\int_0^\infty\int_{\mathbb{T}}\left(p+\frac{\partial}{\partial v}\right)e_1^{n+1}f_1f_2f_3dxdt\\
& &+\int_0^\infty\int_{\mathbb{T}}e_1^{n+1}f_1f_2f_3'dxdt\\\nonumber
&=&\frac{f_3(L_2)}{f_4(L_2)}\left(p+\frac{\partial}{\partial v}\right)\Phi_2^{n}(x,L_2),
\end{eqnarray}
and
\begin{eqnarray}\label{RRobinconditions2}
\left(-q+\frac{\partial}{\partial v}\right)\Phi_2^{n+1}(x,L_1)&=&\frac{f_4(L_1)}{f_3(L_1)}\left(-q+\frac{\partial}{\partial v}\right)\Phi_1^{n}(x,L_1).
\end{eqnarray}
By  Theorem $\ref{Wellposed}$, if $(\Phi_1^n,\Phi_2^n)$ is well-defined and belongs to $(H^2(-\infty,L_2), H^2(L_1,\infty))$ then $(\Phi_1^{n+1},\Phi_2^{n+1})$  is well-defined and belongs to $(H^2(-\infty,L_2), H^2(L_1,\infty))$. Consider $(\ref{REqSubdomain1})$ with the index $n$ instead of $n+1$ on $\mathbb{T}\times(-\infty,L_1)$ and take $\varphi_1^n$ in $H^1(-\infty,L_1)$ as a test function, then
\begin{eqnarray*}
0&=&\int_{-\infty}^{L_1}\int_0^\infty\int_{\mathbb{T}}e_1^nf_1f_2f_3\left(-\frac{f_1'}{f_1}-\frac{f_2'}{f_2}-2\left(\frac{f_3'}{f_3}\right)^2+\frac{f_3''}{f_3}\right)\varphi_1^ndxdtdv\\
& &-\int_{-\infty}^{L_1}\frac{\partial^2}{\partial v^2}\Phi_1^n\varphi_1^ndv+\int_{-\infty}^{L_1}2\frac{f_3'}{f_3}\frac{\partial}{\partial v}\Phi_1^n\varphi_1^ndv.
\end{eqnarray*}
This implies
\begin{eqnarray}\label{RGreenleft}
& &\frac{\partial}{\partial v}\Phi_1^n(L_1)\varphi_1^n(L_1)-q\Phi_1^n(L_1)\varphi_1^n(L_1)\\\nonumber
&=&\int_{-\infty}^{L_1}\int_0^\infty\int_{\mathbb{T}}e_1^nf_1f_2f_3\left(-\frac{f_1'}{f_1}-\frac{f_2'}{f_2}-2\left(\frac{f_3'}{f_3}\right)^2+\frac{f_3''}{f_3}\right)\varphi_1^ndxdtdv\\\nonumber
& &+\int_{-\infty}^{L_1}\frac{\partial}{\partial v}\Phi_1^n\frac{\partial}{\partial v}\varphi_1^ndv+\int_{-\infty}^{L_1}2\frac{f_3'}{f_3}\frac{\partial}{\partial v}\Phi_1^n\varphi_1^ndv-q\Phi_1^n(L_1)\varphi_1^n(L_1).
\end{eqnarray}
Considering $(\ref{REqSubdomain2})$ on $\mathbb{T}\times(L_1,\infty)$ and taking $\varphi_2^{n+1}$ in $H^1(L_1,\infty)$ as a test function, we get
\begin{eqnarray}\label{RGreenrigt}
& &-\frac{\partial}{\partial v}\Phi_2^{n+1}(L_1)\varphi_2^{n+1}(L_1)+q\Phi_2^{n+1}(L_1)\varphi_2^{n+1}(L_1)\\\nonumber
&=&\int^{\infty}_{L_1}\int_0^\infty\int_{\mathbb{T}}e_2^{n+1}f_1f_2f_3\left(2\alpha-\frac{f_2'}{f_2}-2\left(\frac{f_3'}{f_3}\right)^2+\frac{f_3''}{f_3}\right)\varphi_2^{n+1}dxdtdv\\\nonumber
& &+\int^{\infty}_{L_1}\frac{\partial}{\partial v}\Phi_2^{n+1}\frac{\partial}{\partial v}\varphi_2^{n+1}dv+\int^{\infty}_{L_1}2\frac{f_3'}{f_3}\frac{\partial}{\partial v}\Phi_2^{n+1}\varphi_2^{n+1}dv+q\Phi_2^{n+1}(L_1)\varphi_2^{n+1}(L_1).
\end{eqnarray}
Equations $(\ref{RRobinconditions2})$, $(\ref{RGreenleft})$ and $(\ref{RGreenrigt})$ imply
\begin{eqnarray*}
& &-\frac{f_4(L_1)}{f_3(L_1)}\left[\int_{-\infty}^{L_1}\int_0^\infty\int_{\mathbb{T}}e_1^nf_1f_2f_3\left(2\alpha-\frac{f_2'}{f_2}-2\left(\frac{f_3'}{f_3}\right)^2+\frac{f_3''}{f_3}\right)\varphi_1^ndxdtdv\right.\\\nonumber
& &\left.+\int_{-\infty}^{L_1}\frac{\partial}{\partial v}\Phi_1^n\frac{\partial}{\partial v}\varphi_1^ndv+\int_{-\infty}^{L_1}2\frac{f_3'}{f_3}\frac{\partial}{\partial v}\Phi_1^ndv-q\Phi_1^n(L_1)\varphi_1^n(L_1)\right]\\\nonumber
&=&\int^{\infty}_{L_1}\int_0^\infty\int_{\mathbb{T}}e_2^{n+1}f_1f_2f_3\left(2\alpha-\frac{f_2'}{f_2}-2\left(\frac{f_3'}{f_3}\right)^2+\frac{f_3''}{f_3}\right)\varphi_2^{n+1}dxdtdv\\\nonumber
& &+\int^{\infty}_{L_1}\frac{\partial}{\partial v}\Phi_2^{n+1}\frac{\partial}{\partial v}\varphi_2^{n+1}dv+\int^{\infty}_{L_1}2\frac{f_3'}{f_3}\frac{\partial}{\partial v}\Phi_2^{n+1}\varphi_2^{n+1}dv+q\Phi_2^{n+1}(L_1)\varphi_2^{n+1}(L_1).
\end{eqnarray*}
In the above equality, choose $\varphi_2^{n+1}$ to be $\Phi_2^{n+1}$, and $\varphi_1^n$ to be the extension of $\Phi_2^{n+1}$ over $(-\infty,L_1)$ such that there exists a constant $C$ satisfying 
$$\|\varphi_1^n\|_{H^1(-\infty,L_1)}\leq C\|\Phi_2^{n+1}\|_{H^1(L_1,\infty)}$$
and 
$$\|\varphi_1^n\|_{L^2(-\infty,L_1)}\leq C\|\Phi_2^{n+1}\|_{L^2(L_1,\infty)},$$
to get
\begin{eqnarray}\label{Rbalance}\nonumber
& &-\frac{f_4(L_1)}{f_3(L_1)}\left(\int_{-\infty}^{L_1}\int_0^\infty\int_{\mathbb{T}}e_1^nf_1f_2f_3\left(2\alpha-\frac{f_2'}{f_2}-2\left(\frac{f_3'}{f_3}\right)^2+\frac{f_3''}{f_3}\right)\varphi_1^ndxdtdv\right.\\
& &\left.+\int_{-\infty}^{L_1}\frac{\partial}{\partial v}\Phi_1^n\frac{\partial}{\partial v}\varphi_1^ndv+\int_{-\infty}^{L_1}2\frac{f_3'}{f_3}\frac{\partial}{\partial v}\Phi_1^n\varphi_1^ndv-q\Phi_1^n(L_1)\varphi_1^n(L_1)\right)\\\nonumber
&=&\int^{\infty}_{L_1}\int_0^\infty\int_{\mathbb{T}}e_2^{n+1}f_1f_2f_3\left(2\alpha-\frac{f_2'}{f_2}-2\left(\frac{f_3'}{f_3}\right)^2+\frac{f_3''}{f_3}\right)\Phi_2^{n+1}dxdtdv\\\nonumber
& &+\int^{\infty}_{L_1}\left|\frac{\partial}{\partial v}\Phi_2^{n+1}\right|^2dv+\int^{\infty}_{L_1}2\frac{f_3'}{f_3}\frac{\partial}{\partial v}\Phi_2^{n+1}\Phi_2^{n+1}dv+q(\Phi_2^{n+1}(L_1))^2.
\end{eqnarray}
We now bound the right hand side of $(\ref{Rbalance})$ from below and the left hand side of $(\ref{Rbalance})$ from above. According to $(\ref{assumptionf3f4})$, the right hand side of $(\ref{Rbalance})$ is greater than or equal to
\begin{eqnarray}\label{Rbalance1}\nonumber
& &\int^{\infty}_{L_1}\alpha|\Phi_2^{n+1}|^2dv+\int^{\infty}_{L_1}\left|\frac{\partial}{\partial v}\Phi_2^{n+1}\right|^2dv+\int^{\infty}_{L_1}2\frac{f_3'}{f_3}\frac{\partial}{\partial v}\Phi_2^{n+1}\Phi_2^{n+1}dv\\\nonumber
&\geq &\int^{\infty}_{L_1}\alpha|\Phi_2^{n+1}|^2dv+\int^{\infty}_{L_1}\left|\frac{\partial}{\partial v}\Phi_2^{n+1}\right|^2dv-\left\|\frac{f_3'}{f_3}\right\|_{L^\infty(\mathbb{R})}\int^{\infty}_{L_1}\left(\epsilon\left|\frac{\partial}{\partial v}\Phi_2^{n+1}\right|^2+\frac{1}{\epsilon}|\Phi_2^{n+1}|^2\right)dv\\
&\geq &\int^{\infty}_{L_1}\frac{\alpha}{2}|\Phi_2^{n+1}|^2dv+\frac{1}{2}\int^{\infty}_{L_1}\left|\frac{\partial}{\partial v}\Phi_2^{n+1}\right|^2dv,
\end{eqnarray}
for $\alpha$ large enough. 
\\ Again due to $(\ref{assumptionf3f4})$, the left hand side of $(\ref{Rbalance})$ is less than or equal to
\begin{eqnarray}\label{Rbalance2}
& &\frac{f_4(L_1)}{f_3(L_1)}\left(\int_{-\infty}^{L_1}\alpha|\Phi_1^n||\varphi_1^n|dv+\int_{-\infty}^{L_1}\left|\frac{\partial}{\partial v}\Phi_1^n\right|\left|\frac{\partial}{\partial v}\varphi_1^n\right|dv\right.\\\nonumber
& &\left.+\int_{-\infty}^{L_1}2\left|\frac{f_3'}{f_3}\right|\left|\frac{\partial}{\partial v}\Phi_1^n\right||\varphi_1^n|dv+q|\Phi_1^n(L_1)||\varphi_1^n(L_1)|\right)\\\nonumber
& \leq& \frac{f_4(L_1)}{f_3(L_1)}C^*\left(\int_{-\infty}^{L_2}{\alpha}|\Phi_1^{n}|^2dv+\int_{-\infty}^{L_2}\left|\frac{\partial}{\partial v}\Phi_1^{n}\right|^2dv\right)\\\nonumber
& &+\frac{f_4(L_1)}{f_3(L_1)}C^*\left(\int_{-\infty}^{L_2}{\alpha}|\Phi_2^{n+1}|^2dv+\int_{-\infty}^{L_2}\left|\frac{\partial}{\partial v}\Phi_2^{n+1}\right|^2dv\right)\\\nonumber
& &\leq \frac{1}{8}\left(\int_{-\infty}^{L_2}{\alpha}|\Phi_1^{n}|^2dv+\int_{-\infty}^{L_2}\left|\frac{\partial}{\partial v}\Phi_1^{n}\right|^2dv\right)\\\nonumber
& &+\frac{1}{8}\left(\int_{-\infty}^{L_2}{\alpha}|\Phi_2^{n+1}|^2dv+\int_{-\infty}^{L_2}\left|\frac{\partial}{\partial v}\Phi_2^{n+1}\right|^2dv\right),
\end{eqnarray}
where $C^*$ is some constant not depending on $\alpha$.
\\ Compare the two inequalities $(\ref{Rbalance1})$ and $(\ref{Rbalance2})$ 
\begin{eqnarray*}\label{Rbalance3}
\int_{-\infty}^{L_2}{\alpha}|\Phi_1^{n+1}|^2dv+\int_{-\infty}^{L_2}\left|\frac{\partial}{\partial v}\Phi_1^{n+1}\right|^2dv\leq\frac{1}{2}\left(\int^{\infty}_{L_1}{\alpha}|\Phi_2^{n}|^2dv+\int^{\infty}_{L_1}\left|\frac{\partial}{\partial v}\Phi_2^{n}\right|^2dv\right).
\end{eqnarray*}
Similarly
\begin{eqnarray*}\label{Rbalance4}
\int^{\infty}_{L_1}{\alpha}|\Phi_2^{n+1}|^2dv+\int^{\infty}_{L_1}\left|\frac{\partial}{\partial v}\Phi_2^{n+1}\right|^2dv\leq\frac{1}{2}\left(\int_{-\infty}^{L_2}{\alpha}|\Phi_1^{n}|^2dv+\int_{-\infty}^{L_2}\left|\frac{\partial}{\partial v}\Phi_1^{n}\right|^2dv\right).
\end{eqnarray*}
Take the sum of the previous two inequalities to get
\begin{eqnarray*}\label{Contraction}
& &\int_{-\infty}^{L_2}{\alpha}|\Phi_1^{n+1}|^2dv+\int_{-\infty}^{L_2}\left|\frac{\partial}{\partial v}\Phi_1^{n+1}\right|^2dv+\int^{\infty}_{L_1}{\alpha}|\Phi_2^{n+1}|^2dv+\int^{\infty}_{L_1}\left|\frac{\partial}{\partial v}\Phi_2^{n+1}\right|^2dv\\\nonumber
&\leq&\frac{1}{2}\left(\int_{-\infty}^{L_2}{\alpha}|\Phi_1^{n}|^2dv+\int_{-\infty}^{L_2}\left|\frac{\partial}{\partial v}\Phi_1^{n}\right|^2dv+\int^{\infty}_{L_1}{\alpha}|\Phi_2^{n}|^2dv+\int^{\infty}_{L_1}\left|\frac{\partial}{\partial v}\Phi_2^{n}\right|^2dv\right).
\end{eqnarray*}
The conclusion of the Theorem follows by letting $n$ tend to $\infty$ in the previous estimate.
\end{proof}
\section{Numerical experiments}\label{numer}

In this section we provide some numerical tests to support the theoretical analysis of the previous sections. 

\subsection{Model problem}
We consider the initial boundary value problem 

\begin{equation}
\label{Model_problem}
\begin{aligned}
\dfrac{\partial u}{\partial t} + v \dfrac{\partial u}{\partial x} - \dfrac{\partial^2 u}{\partial v^2}  &= f \quad &\mbox{in } (0,T)\times [0,1]\times [-1,1]\\ \\
u(t,0,v) &= u(t,1,v)  &\mbox{on } (0,T)\times [-1,1]\\
\dfrac{\partial u}{\partial v} (t,x,-1) &= 0 &\mbox{on } (0,T)\times [0,1]\\
\dfrac{\partial u}{\partial v} (t,x,1) &= 0 &\mbox{on } (0,T)\times [0,1]\\
\end{aligned}
\end{equation}
equipped with homogeneous Neumann boundary conditions in $v$ and periodic boundary conditions in $x$.
We claim that different choices of boundary condition in $v=-1$ and $v=1$ do not affect the results we show in what follows. 
Since the problem is linear, we can directly test the convergence on the error equation ({\it i.e.} letting $f\equiv 0$) whose unknown, with a little abuse of notation, 
we still denote by $u$.

\subsection{Finite dimensional approximation on a single domain}

We briefly describe here the numerical approximation of equation \eqref{Model_problem}, and we focus for presentation purposes on a single domain.
We discretize equation \eqref{Model_problem} by an operator splitting technique (see e.g. \cite{Marchuk}), where we first solve a parabolic problem 
in $(t,v)$ for half the time step, and we correct it by explicitly advancing the transport part of the equation in $(t,x)$. Let then $\Delta t$ be the time 
discretization step, and let $\tau=\Delta t/2$. \\

\noindent
{\bf Step 1.} \quad Solve, in $[t,t+\tau]$, for all $x\in[0,1]$, \quad $\dfrac{\partial}{\partial t} w(t,x,v) -\dfrac{\partial^2 }{\partial v^2}w (t,x,v)= 0$. \\

\bigskip
\noindent
{\bf Step 2.} \quad For all $x\in[0,1]$, \quad  $u(t+\Delta t,x,v) = w (t,x-\tau v,v).$ \\

\bigskip
\noindent
We discretize the parabolic part of equations \eqref{Model_problem} with an implicit Euler scheme in $t$, and by finite elements in the $v$ direction 
(see e.g. \cite{QV:1994:NAPDE}). The transport part is solved explicitly by interpolation on the solution computed at Step 1. 
We denote by $h_x$ and $h_v$ the discretization steps in the $x$ and  $v$ variable, respectively, and by $N_x$ and $2N_v$ the corresponding 
grid point numbers. We let $x_m=m\,h_x$ ($m=0,..,N_x-1$), $v_i=-1+i\,h_v$ ($i=0,..,2N_v-1$), we denote by $\{\phi_j\}_{j=0,.., 2 N_v-1}$ 
a nodal basis for the finite element space associated to $v$, and we can approximate $u(t^n,x_m,v)$ by
$$
u(t^n,x_m,v) \ \sim\ u_m (t^n,v)\, = \, \sum _{j=0}^{N_v} u_{j,m}(t^n) \phi_j(v).
$$ 
For the sake of compactness in notations, for all $m=0,..,N_x$, we let $\vec{u}_m(t)=[u_{1,m}(t),...,u_{2N_v,m}(t)]^T$ and $\vec{u}_m^n = \vec{u}_m(t^n)$. \\

\noindent
The numerical approximation of \eqref{Model_problem} is then computed by the following operator splitting scheme.\\

\noindent
Given $\left\{{u}^{n}_{i,m}\right\}_{i=1,..,2N_v, m=1,..,N_x}$ \\

\noindent
{\bf Step 1.} \quad For $m=0,..., N_x-1$, solve

\begin{equation}
\label{step1}
\dfrac{1}{\tau}\,M \vec{u}^{n+1/2}_m + S  \vec{u}^{n+1/2}_m = \dfrac{1}{\tau}\,M  \vec{u}^{n}_m,
\end{equation}
where $M$ and $S$ are the mass and stiffness matrices, whose entries $(i,j)$ are given by   
\begin{equation}
\label{MMvS}
[M]_{ij}=\int_0^1 \phi_j\phi_i\,dv
\qquad \qquad
[S]_{ij}=\int_0^1 \dfrac{d \phi_j}{d v} \dfrac{d \phi_i}{d v}\,dv. 
\end{equation}
Let then $\vec{u}^{n+1/2}_m=\left[{u}^{n+1/2}_{1,m}, ..., {u}^{n+1/2}_{2 N_v,m}\right]^T$. \\

\noindent
{\bf Step 2.} \quad For $i=0,..,N_v-1$, set 

\begin{equation}
\label{step2_a} 
\begin{aligned}
{u}^{n+1}_{i,m} & =  \left(1-|v_i|\,\tau \right){u}^{n+1/2}_{i,m} +  \left(|v_i| \tau\,\right)  {u}^{n+1/2}_{i,m+1} \qquad \qquad & \mbox{for } m=1,...,N_x-1 \\ \\
{u}^{n+1}_{i,N_x} & =  {u}^{n+1}_{i,1}. &
\end{aligned}
\end{equation}

\medskip
\noindent
For $i=N_v,..,2 N_v-1$, set

\begin{equation}
\label{step2_b} 
\begin{aligned}
{u}^{n+1}_{i,1} & =  \left(1-|v_i|\,\tau \right){u}^{n+1/2}_{i,m} +  \left(|v_i| \tau\,\right)  {u}^{n+1/2}_{i,m-1} \qquad \qquad & \mbox{for } m=1,...,N_x-1 \\ \\
{u}^{n+1}_{i,1} & =  {u}^{n+1}_{i,N_x}. &
\end{aligned}
\end{equation}

\begin{remark}
In the numerical tests of the following section, we use linear finite elements and a Cavalieri-Simpson quadrature rule to evaluate these entries.  
Since the Cavalieri-Simpson rule is third order accurate, the matrices $M$ and $S$ are computed exactly. 
A complete stability and convergence analysis of the numerical procedure described here is beyond the scope of this paper and is the object  
of a forthcoming study.
\end{remark}

\subsection{Schwarz Waveform Relaxation}

We decompose the computational domain $\Omega=[0,T]\times [0,1]\times [-1,1]$  into two subdomains
\begin{equation}
\label{Om1_Om2}
\Omega_1 = [0,T]\times [0,1]\times [-1,\beta] \qquad\qquad \qquad \Omega_2 = [0,T]\times [0,1] \times [\alpha,1],
\end{equation}
which may or may not overlap ($\beta-\alpha\geq 0$). As a matter of fact, even if the analysis was carried on in the case of  
overlapping subdomains only, the use of Robin interface conditions in an Optimized Schwarz Waveform Relaxation (OSWR) algorithm
guarantees convergence also in the absence of overlap, a feature not shared by the Classical Schwarz Waveform Relaxation (CSWR)
one. In what follows we denote by $L=\beta-\alpha$ the size of the overlap between the two subdomains. \\
We introduce the interface variables 
\begin{equation}
\label{lambda12}
\lambda_1(t,x,\beta) = \mathcal{Q}_1\, u_2(t,x,\beta)  \qquad \qquad  \lambda_2(t,x,0) = \mathcal{Q}_2\, u_1(t,x,0),
\end{equation}
where the operators $\mathcal{Q}_1$ and $\mathcal{Q}_2$ are given by 
$$
\mathcal{Q}_1\, w\, = \,w \qquad \qquad \qquad \mathcal{Q}_2 \, w\, =\, w
$$
for the CSWR, and by
$$
\mathcal{Q}_1\, w = \left(p+\dfrac{\partial}{\partial v}\right)\, w \qquad \qquad \mathcal{Q}_2\, w = \left(q-\dfrac{\partial}{\partial v}\right)\, w
$$
for the OSWR. With these positions, the Schwarz Waveform Relaxation algorithms read as follows.\\

Given $\lambda_1^0(t,x,\beta)$ on $[0,T]\times[0,1]$, solve for $k\geq 1$ until convergence

\begin{equation}
\label{SWR1}
\begin{aligned}
\dfrac{\partial u_1^k}{\partial t} + v \dfrac{\partial u_1^k}{\partial x} - \dfrac{\partial^2 u_1^k}{\partial v^2}  &= 0 \qquad  &\mbox{in }  \Omega_1\ \\
u_1^k(t,0,v) &= u_1^k(t,1,v) \qquad &\mbox{on } [0,1]\ \\
\dfrac{\partial u_1^k}{\partial v} (t,x,-1) &= 0  \qquad &\mbox{on } [0,T]\times [0,1]\ \\ 
\mathcal{Q}_1\, u_1^k(t,x,\beta) &=\lambda^{k-1}_1(t,x,\beta) \qquad &\mbox{on } [0,T]\times[0,1], 
\end{aligned}
\end{equation}

\begin{equation}
\label{SWR2}
\lambda_2^k(t,x,\alpha) = \mathcal{Q}_2\, u_1^{k}(t,x,\alpha) \qquad \qquad \qquad \qquad  \mbox{on } [0,T]\times[0,1], 
\end{equation}

\begin{equation}
\label{SWR3}
\begin{aligned}
\dfrac{\partial u_2^k}{\partial t} + v \dfrac{\partial u_2^k}{\partial x} - \dfrac{\partial^2 u_2^k}{\partial v^2}  &= 0 \qquad & \mbox{in }  \Omega_2\ \\
u_2^k(t,0,v) &= u_2^k(t,1,v) \qquad &\mbox{on } [0,T]\times [0,1]\ \\
\dfrac{\partial u_2^k}{\partial v} (t,x,1)&= 0 \qquad &\mbox{on } [0,T]\times 0,1],  \\
\mathcal{Q}_2\, u_2^k(t,x,\alpha) &= \lambda_2^{k}(t,x,\alpha) \qquad &\mbox{on } [0,T]\times [0,1]
\end{aligned}
\end{equation}

\begin{equation}
\label{SWR4}
\lambda_1^{k}(t,x,\beta) = \mathcal{Q}_1\, u_2^k(t,x,\beta) \qquad \qquad \qquad \qquad \mbox{on } [0,T]\times [0,1]. 
\end{equation}

For a given tolerance $\varepsilon>0$, the Schwarz Waveform Relaxation algorithm \eqref{SWR1}-\eqref{SWR4} is considered to have reached convergence when 
\begin{equation}
\label{error}
\left\| u^k_1(t,x,v)-u^k_2(t,x,v)  \right\|_{L^\infty([0,T]\times[0,1])\times (\alpha,\beta)}\ < \ \varepsilon.
\end{equation}
 
\begin{remark}
The Schwarz waveform relaxation algorithm is serial in the form presented in \eqref{SWR1}-\eqref{SWR4}, but it can be easily parallelized 
by just replacing $\lambda_2^{k}(t,x,\alpha)$ with $\lambda_2^{k-1}(t,x,\alpha)$ in \eqref{SWR3}.

\end{remark}

\subsection{Optimization of the Robin parameters}
Since an analytical optimization of the Robin parameters $(p,q)$ is not available, we perform an empirical optimization both in the case of 
one-sided ($p=q$) and two-sided ($p\neq q$) interface conditions. We let $T=2$, and for the linearity of the problem we test directly the convergence
on the error equation. We discretize the domains $\Omega_1$ and $\Omega_2$ by a uniform grid. 
Since the mesh size in $v$ is not affecting the size of the interface problem, we use the same step $h_v$ in both $\Omega_1$ and $\Omega_2$,  
with $h_{v}=h_x=\Delta t=0.01$. As a consequence, the interface problems features $20,200$ unknowns. We choose an overlap of three elements 
($L=3 h_v$). We initialize the interface variable with a random value for $\lambda_1^0(t, x, \beta)$, in order to have all the frequencies represented 
in the initial error. Finally, we consider the algorithm to have converged when the error \eqref{error} drops below $\varepsilon=10^{-6}$.

\subsubsection{One-sided Optimized Schwarz Waveform Relaxation: OSWR(p)}
In Figure 1 (left) we plot the iteration counts needed to achieve convergence, as the parameter $p$ varies. In Figure 1 (right) 
we plot the error after 15 iteration for different values of $p$. The optimal parameter is numerically identified as $p^*=4.23$, by sampling the interval $(4,5)$
with step 0.001. Although the iteration counts is the same as for $p=4$, the Robin parameter $p^*$ features a steeper convergence history. This is the case also 
for $p=5$, which requires 2 more iterations to converge, but has a smaller error than $p=4$ after 15 iterations. 

\begin{figure}[!h]
\label{Figure:p}
\begin{center}
\includegraphics[width=0.4\textwidth]{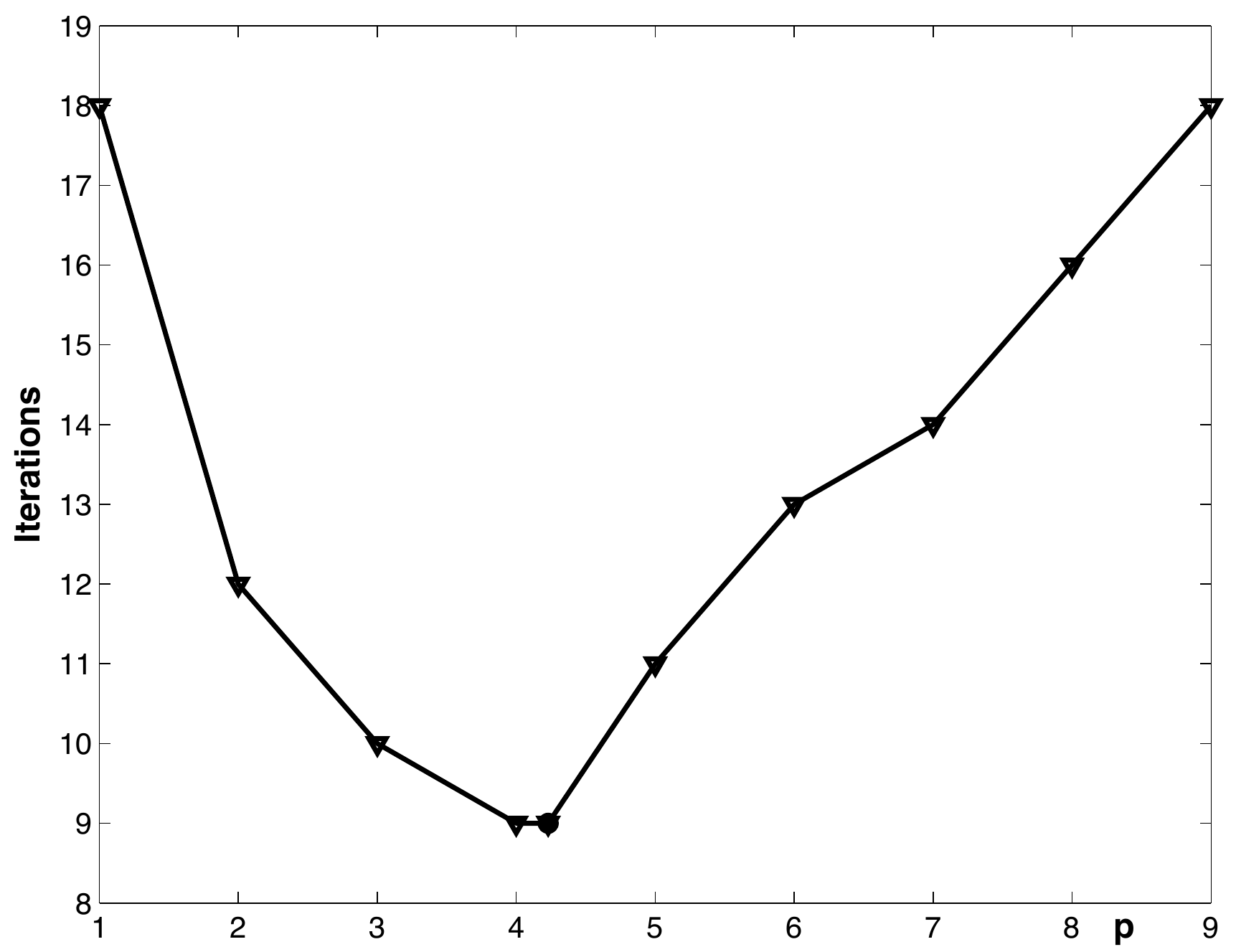} \qquad
\includegraphics[width=0.4\textwidth]{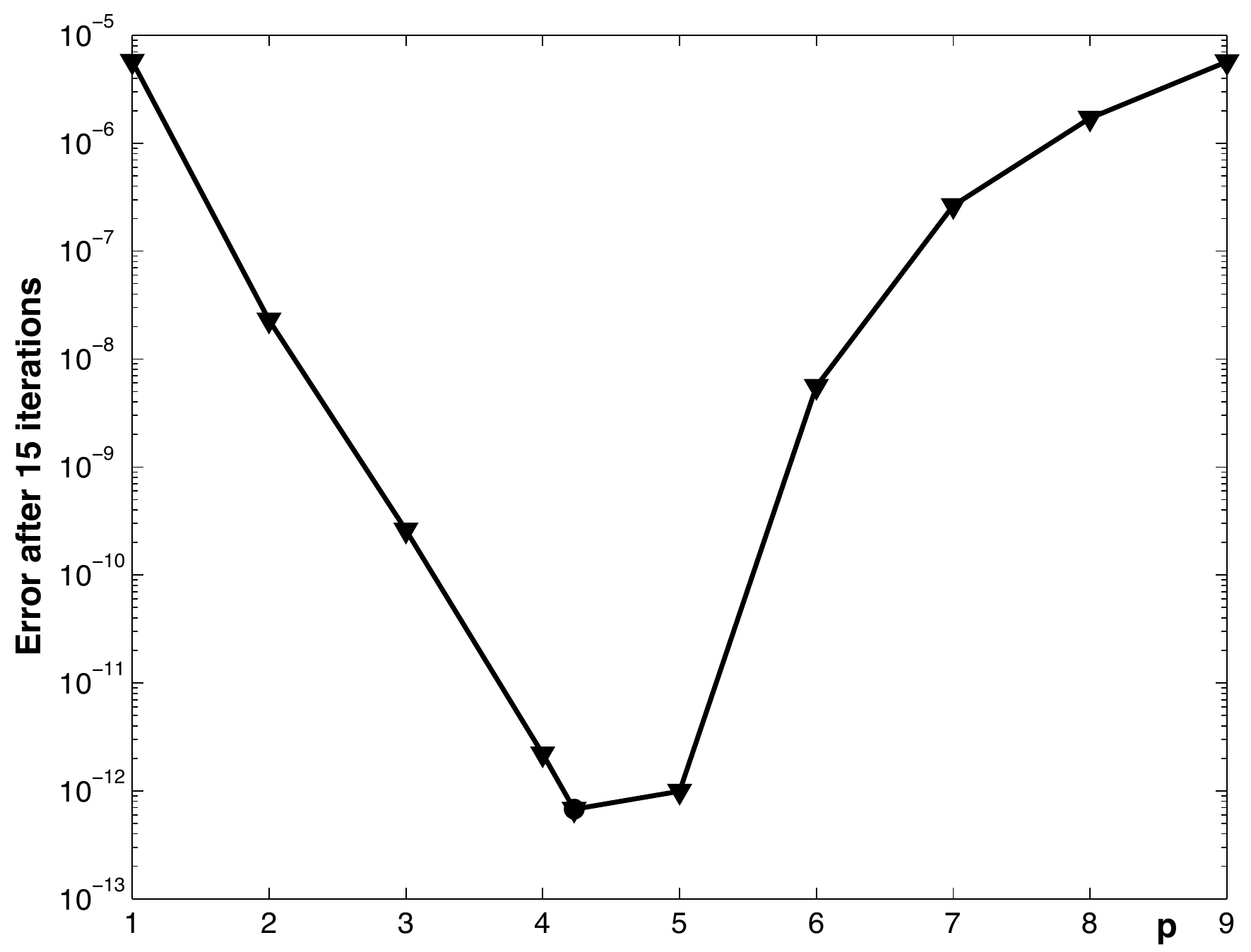}
\end{center}
\caption{OSWR(p). Left: iteration counts  to reach convergence as a function of the Robin parameter $p$. Right: error 
after 15 iterations as a function of  $p$.} 
\end{figure}
\noindent

\subsubsection{Optimized two-sided Schwarz Waveform Relaxation: OSWR(p,q)}
In Figure 2 (left) we plot the iteration counts needed to achieve convergence, as the parameters $p$ and $q$ vary. In Figure 2 (right) 
we plot the error after 15 iteration for different values of $p$ and $q$.

\begin{figure}[!h]
\label{Figure:iter_p4}
\begin{center}
\includegraphics[width=0.4\textwidth]{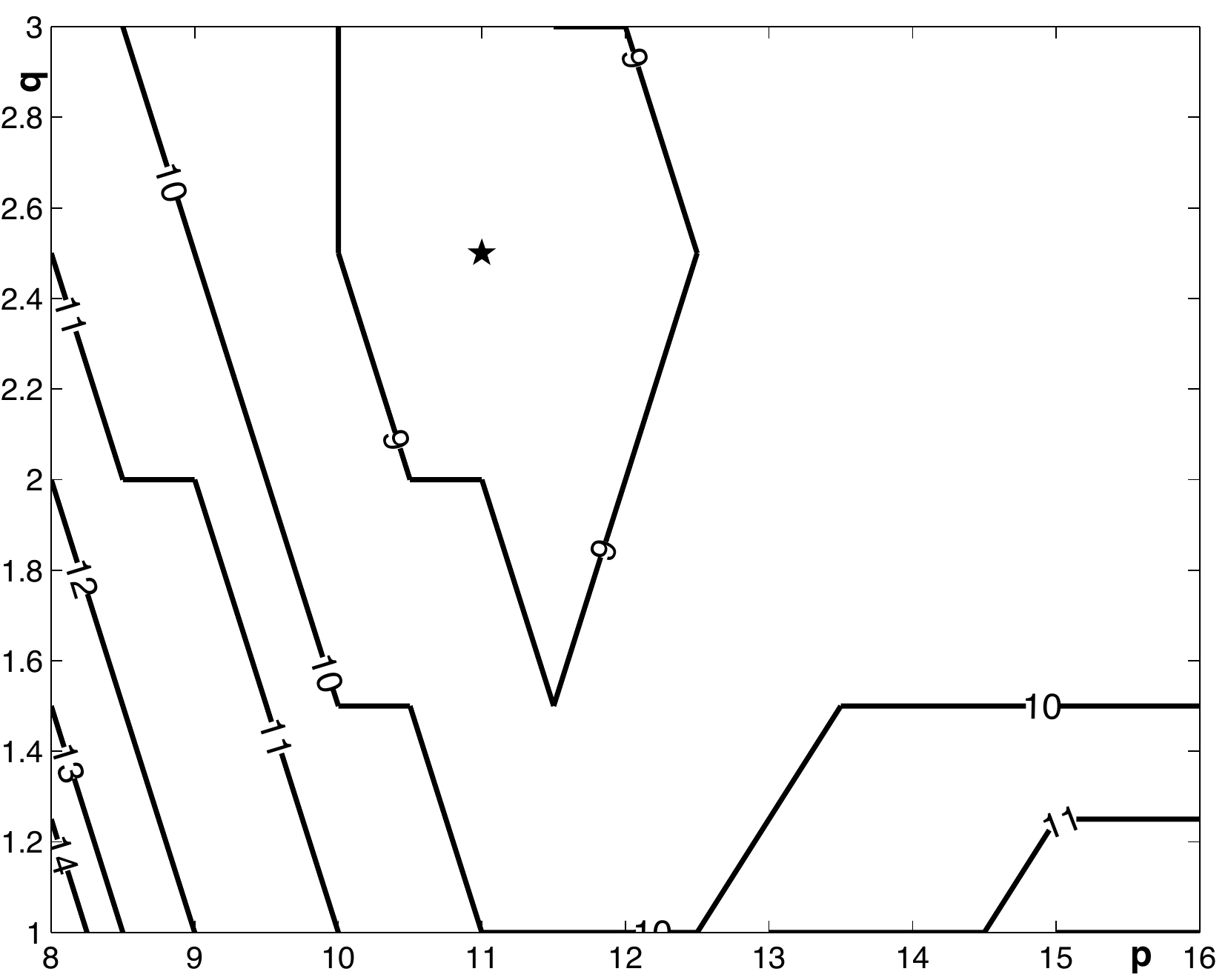} \qquad
\includegraphics[width=0.4\textwidth]{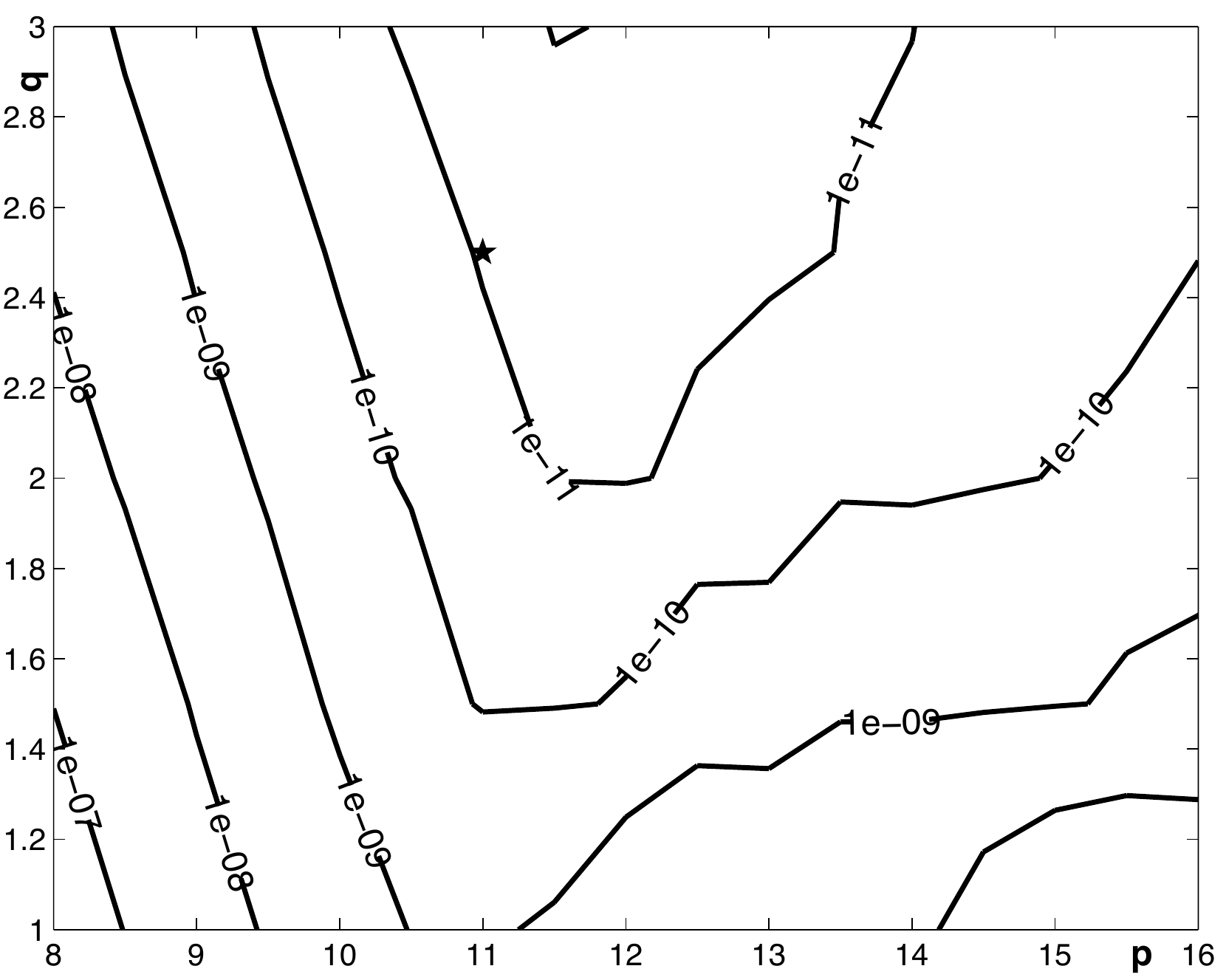}
\end{center}
\caption{OSWR(p,q). Left: iteration counts  to reach convergence as a function of the Robin parameters $p$ and $q$. Right: error 
after 15 iterations as a function of  $(p,q)$.} 
\end{figure}

\subsection{Comparison between Optimized and Classical Schwarz Waveform Relaxation}
We compare in this section the performance of the Classical and Optimized algorithms. 
We consider both non-overlapping and overlapping decompositions as in \eqref{Om1_Om2}, always with the overlap of thre element ($L=3 h_v$). 
Following the results from the previous Section, we implemented OSWR(p) with $p=4.23$, and with 
OSWR(p,q) with $p=11$ and $q=2.5$. We consider a reference mesh size $\Delta t=h_{v}=h_x=0.01$, and test the behavior of 
the algorithm in four successive dyadic mesh refinements, $\tau_j=2^{-j} \times$ 0.01 ($\tau= \Delta t, h_x, h_v$), with $j=0,..,4$. 
We report the results in Table 1.\\
In the overlapping case, both OSWR(p) and OSWR(p,q) algorithm appear to be almost insensitive to the mesh refinement, while
the CSWR appears to be very sensitive to it. The two-sided OSWR(p,q) appears globally more robust in terms 
of iteration counts with respect to the one-sided OSWR(p), whose iteration counts still remain more than reasonable. Both algorithms outperform the CSWR.\\
In the non-overlapping case, a similar pattern is observed for OSWR(p) and OSWR(p,q). Both algorithms appear to be a little sensitive to the size of the 
interface problem. However, iteration counts are higher than in the overlapping case, but not significantly higher. 
The OSWR(p,q) is more robust than the OSWR(p), featuring an increase of around 50\% in iterations for the most refined case, while the latter experiences 
a doubling. For both algorithms, however, the iteration counts remain reasonable in all cases. Finally, as expected, CSWR does not converge in the absence of 
overlap. Finally, we plot in Figure 3 the convergence history of the three overlapping algorithms at level $j=2$ of refinement.

\begin{table}[htp!]
\begin{center}
\begin{tabular}{|l||c|c|c|c|c||}
\hline 
$\Delta_t=2^{-j}\times$0.01&\multicolumn{5}{|c||}{Overlapping} 
\\
$h_x=2^{-j}\times$0.01&\multicolumn{5}{|c||}{ ($L = 3\times h_v$)} 
\\
\cline{2-6} 
$h_v=2^{-j}\times$0.01&  $j=0$& $j=1$ & $j=2$ & $j=3$& $j=4$ \\
\hline
{\bf CSWR} & 70 & 105 & 132 & $>$150 & $>$150\\
\hline 
{\bf OSWR(p)}& 9 & 12 & 15 & 17 &  18\\
\hline 
{\bf OSWR(p,q)}& 9 & 10 & 10 & 10 & 13\\
\hline 
\hline
$\Delta_t=2^{-j}\times$0.01&\multicolumn{5}{|c||}{Non-overlapping} 
\\
$h_x=2^{-j}\times$0.01&\multicolumn{5}{|c||}{ ($L = 0$)} 
\\
\cline{2-6} 
$h_v=2^{-j}\times$0.01&  $j=0$& $j=1$ & $j=2$ & $j=3$& $j=4$ \\\hline
{\bf CSWR} & -& -& -& -& -\\
\hline 
{\bf OSWR(p)}& 12& 17 & 20 & 23 & 26\\
\hline 
{\bf OSWR(p,q)}& 11& 12 & 13 & 14 &16\\
\hline 
\end{tabular}
\caption{Classical vs Optimized Schwarz Waveform Relaxation: iteration counts to achieve convergence for successive dyadic refinements. Overlapping
($L=3 h_v$), and non-overlapping decomposition ($L=0$).}
\label{refinement}
\end{center}
\end{table}
\begin{figure}[!h]
\begin{center}
\includegraphics[width=0.6\textwidth]{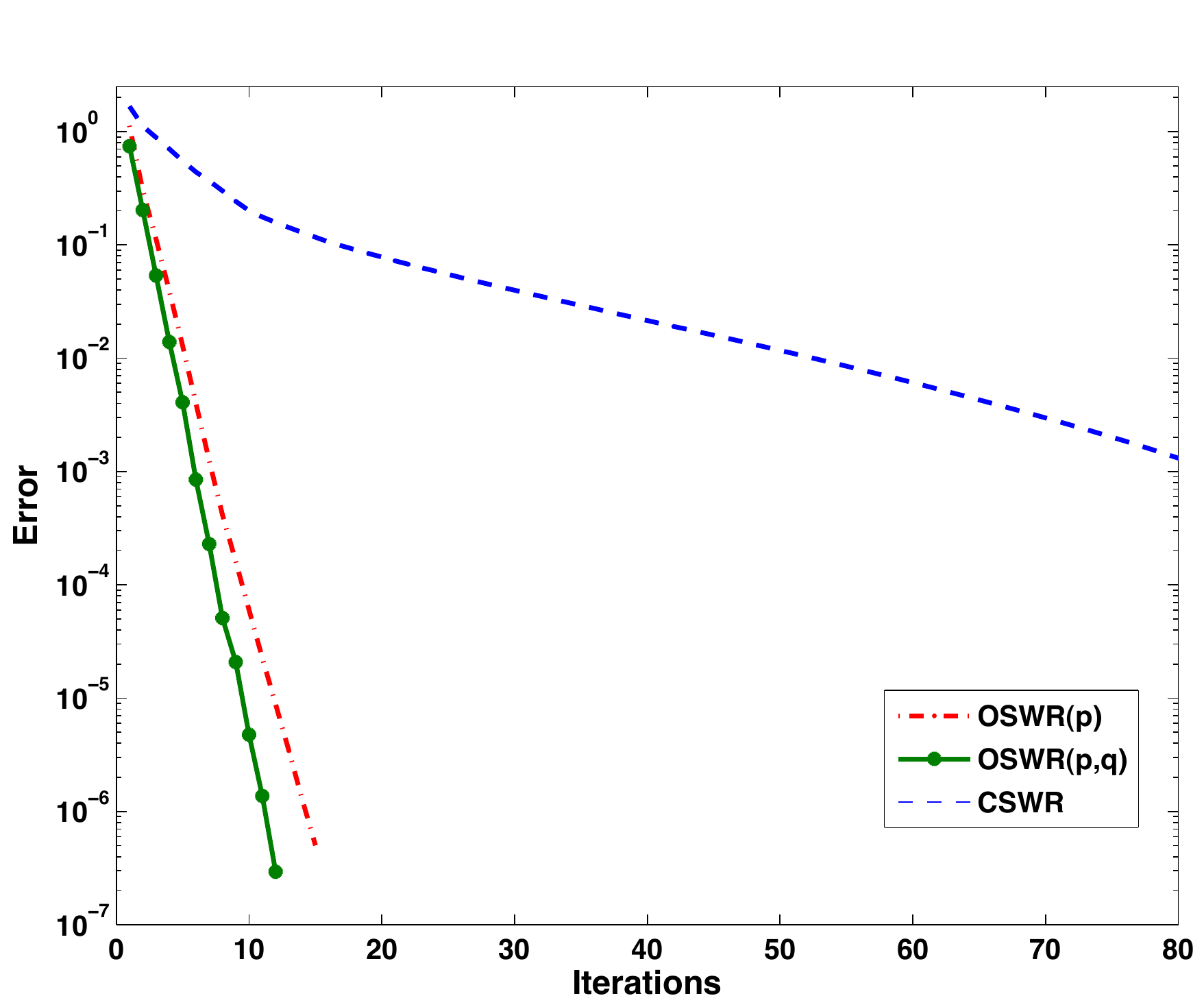}
\end{center}
\caption{Overlapping Schwarz Waveform Relaxation. Convergence history for the three different algorithm, CSWR (blue dashed line), OSWR(p) 
(red dot-dashed line), and OSWR(p,q) (green solid line).} 
\end{figure}

\section{Conclusion}\label{secconclu}
We have designed some new Schwarz waveform relaxation algorithms adapted to the context of the Kolmogorov equations. The domain is split in the $v$-direction, which is the 'parabolic' direction of the equation. The algorithms are proven to be well-posed, stable and useful in both numerical and theoretical senses. The Kolmogorov operator is hypoelliptic and it has properties of both hyperbolic and parabolic operators. Domain decomposition methods for hyperbolic problems are sometimes unstable, even for optimized algorithms, which means that the hyperbolicity of the operator really affects the convergence rates of the algorithm. In our situation, the algorithms are stable in both cases: classical and optimized algorithms.  The theoretical and numerical results in this paper show that the equation is more parabolic than hyperbolic, in the regime of domain decomposition. Moreover, according to our results, the Schwarz waveform relaxation algorithms for the Kolmogorov equation have almost the same properties with an advection diffusion equation or a heat equation. 
\\ {\bf Acknowledgements.} The second author would like to thank his advisor, Professor Enrique Zuazua, for suggesting this topic to him and for his kind and wise guidance. He is also grateful to Professor Jos\'e Antonio Carrillo for fruitful discussions. The second author has been supported by by Grant MTM2011-29306-C02-00, MICINN, Spain, ERC Advanced Grant FP7-246775 NUMERIWAVES, and Grant PI2010-04 of the Basque Government.
\bibliographystyle{plain}\bibliography{Kolmogorov}
\end{document}